\documentclass[twoside]{amsart}

\usepackage{amsmath,amsfonts,amsthm,mathrsfs}
\usepackage{amssymb}

\usepackage[unicode,bookmarks]{hyperref}
\usepackage[alphabetic,initials]{amsrefs}

\usepackage{ifthen}
\usepackage{enumitem}

\newcommand{\Dmod}{\mathcal{D}}
\newcommand{\Mmod}{\mathcal{M}}
\newcommand{\Nmod}{\mathcal{N}}

\newcommand{\decal}[1]{\lbrack #1 \rbrack}

\newcommand{\ltriangle}[4][]%
{\begin{diagram}[#1]%
	{#2} &\rTo& {#3} &\rTo& {#4} &\rTo& {#2 \decal{1}}%
\end{diagram}}

\newcommand{\shH}{\mathcal{H}}

\newcommand{\tensor}{\otimes}

\newcommand{\ZZ}{\mathbb{Z}}

\newcommand{\CC}{\mathbb{C}}

\newcommand{\PPn}[1]{\mathbb{P}^{#1}}

\newcommand{\menge}[2]{\bigl\{ \thinspace #1 \thinspace\thinspace \big\vert%
\thinspace\thinspace #2 \thinspace \bigr\}}

\DeclareMathOperator{\rk}{rk}

\DeclareMathOperator{\Spec}{Spec}

\DeclareMathOperator{\sing}{sing}

\newcommand{\define}[1]{\emph{#1}}

\newcommand{\shf}[1]{\mathscr{#1}}
\newcommand{\OX}{\shf{O}_X}
\newcommand{\OU}{\shf{O}_U}

\newcommand{\argbl}{-}

\newcommand{\into}{\hookrightarrow}
\newcommand{\onto}{\to\hspace{-0.7em}\to}

\newcommand{\shHO}{\shH_{\shf{O}}}

\newcommand{\shV}{\mathcal{V}}

\newcommand{\shVO}{\shV_{\shf{O}}}

\newcommand{\pu}{p^{\ast}}

\newcommand{\IH}{\mathit{IH}}

\newcommand{\shF}{\shf{F}}

\newcommand{\shO}{\shf{O}}

\usepackage{chngcntr}

\makeatletter
\let\@@seccntformat\@seccntformat
\renewcommand*{\@seccntformat}[1]{%
  \expandafter\ifx\csname @seccntformat@#1\endcsname\relax
    \expandafter\@@seccntformat
  \else
    \expandafter
      \csname @seccntformat@#1\expandafter\endcsname
  \fi
    {#1}%
}
\newcommand*{\@seccntformat@subsection}[1]{%
  \textbf{\csname the#1\endcsname.}
}
\makeatother

\makeatletter
\let\@paragraph\paragraph
\renewcommand*{\paragraph}[1]{%
	\vspace{0.3\baselineskip}%
	\@paragraph{\textit{#1}}%
}
\makeatother

\counterwithin{equation}{subsection}
\counterwithout{subsection}{section}
\counterwithin{figure}{subsection}

\newtheorem{theorem}[equation]{Theorem}
\newtheorem*{theorem*}{Theorem}
\newtheorem{lemma}[equation]{Lemma}
\newtheorem*{lemma*}{Lemma}

\newtheorem{proposition}[equation]{Proposition}
\newtheorem*{proposition*}{Proposition}
\newtheorem{conjecture}[equation]{Conjecture}

\theoremstyle{definition}
\newtheorem{definition}[equation]{Definition}
\newtheorem*{definition*}{Definition}
\theoremstyle{remark}
\newtheorem*{remark}{Remark}

\newtheorem*{question}{Question}

\newtheorem{example}[equation]{Example}
\newtheorem*{example*}{Example}

\theoremstyle{plain}

\makeatletter
\let\old@caption\caption
\renewcommand*{\caption}[1]{%
	\setcounter{figure}{\value{equation}}%
	\stepcounter{equation}%
	\old@caption{#1}\relax%
}
\makeatother

\newcounter{thmA}
\newtheorem{theorem-intro}[thmA]{Theorem}

\newcommand{\pp}{\mathfrak{p}}
\newcommand{\mm}{\mathfrak{m}}
\DeclareMathOperator{\Mat}{Mat}

\newcommand{\nut}{\tilde{\nu}}
\newcommand{\sigmat}{\tilde{\sigma}}
\newcommand{\It}{\tilde{I}}
\newcommand{\Mmodt}{\tilde{\Mmod}}
\newcommand{\Nmodt}{\tilde{\Nmod}}
\renewcommand{\Dmod}{\mathscr{D}}
\newcommand{\shT}{\mathscr{T}}

\begin{document}

%========================================================
\title{The zero locus of the infinitesimal invariant}

\author[G.~Pearlstein]{Greg Pearlstein}
\address{%
Department of Mathematics \\
Michigan State University \\
East Lansing, MI 48824
}
\email{gpearl@math.msu.edu}

\author[Ch.~Schnell]{Christian Schnell}
\address{%
	Institute for the Physics and Mathematics of the Universe \\
	The University of Tokyo \\
	5-1-5 Kashiwanoha, Kashiwa-shi \\
	Chiba 277-8583, Japan
}
\email{christian.schnell@ipmu.jp}

\begin{abstract}
Let $\nu$ be a normal function on a complex manifold $X$. The infinitesimal 
invariant of $\nu$ has a well-defined zero locus inside the tangent bundle 
$TX$. When $X$ is quasi-projective, and $\nu$ is admissible, we show that this 
zero locus is constructible in the Zariski topology.
\end{abstract}
\maketitle
%========================================================

\section{Introduction}

\subsection{Main result}

Let $\shH$ be a variation of Hodge structure of weight $-1$ on a Zariski
open subset of a smooth complex projective variety $X$. We shall assume that 
$\shH$ is polarizable and defined over $\ZZ$. We denote the Hodge filtration on
the underlying flat vector bundle $\shHO$ by the symbol $F^{\bullet} \shHO$. Let 
$\nu$ be a normal function, that is to say, a holomorphic and horizontal 
section of the family of intermediate Jacobians $J(\shH)$. For any local 
lifting $\nut$ to a holomorphic section of $\shHO$, we have
\[
	\nabla \nut \in \Omega_X^1 \tensor_{\OX} F^{-1} \shHO,
\]
which is independent of the choice of lifting modulo $\nabla(F^0 \shHO)$. We are
interested in the subset of the tangent bundle $TX$ defined by the condition 
$\nabla \nut \in \nabla(F^0 \shHO)$. Concretely, this is the set
\[
	I(\nu) = \menge{(x, \xi) \in TX}%
		{\text{$\nabla_{\xi} (\nut - \sigma)(x) = 0$ for some $\sigma 
\in F^0 \shHO$}}.
\]
The following theorem describes the structure of $I(\nu)$ for admissible $\nu$.

\begin{theorem} \label{thm:main}
Suppose that $\nu$ is an admissible normal function on a Zariski open subset of
a smooth complex projective variety $X$. Then $I(\nu)$ is constructible with 
respect to the Zariski topology on $TX$.
\end{theorem}

Recall that a subset of an algebraic variety is \define{constructible} if it
is a finite union of subsets that are locally closed in the Zariski topology.
The proof of Theorem~\eqref{thm:main} is given in Section~\ref{sec:proof}.
In Section~\ref{sec:GG}, we describe the relationship between this paper and the study
of algebraic cycles via the approach to the Hodge conjecture by
Green--Griffiths~\cite{GG} using singularities of normal functions.

\subsection{Acknowledgments}

G.P.\@ is partially supported by NSF grant DMS-1002625. C.S.\@ is supported by 
the World Premier International Research Center Initiative (WPI Initiative), 
MEXT, Japan, and by NSF grant DMS-1100606. We thank Patrick Brosnan, Matt Kerr 
and James Lewis for helpful discussions.  G.P. also thanks the IHES for partial 
support during the preparation of this manuscript. 

\section{Proof of the theorem}
\label{sec:proof}

\subsection{Algebraic description of the zero locus}

Since $X$ is a projective algebraic variety, it is possible to describe the
zero locus $I(\nu)$ of the infinitesimal invariant of $\nu$ purely in terms of
algebraic objects. In this section, we shall do this by a straightforward classical
argument.

Using resolution of singularities, we may assume without loss of generality that
$\nu$ is an admissible normal function on $X - D$, where $X$ is a smooth projective
variety, and $D \subseteq X$ is a divisor with normal crossings. Let $\shV$ be the
admissible variation of mixed Hodge structure with $\ZZ$-coefficients corresponding
to $\nu$; then $W_{-1} \shV = \shH$ and $W_0 \shV / W_{-1} \shV \simeq \ZZ(0)$ by our
choice of weights. The integrable connection $\nabla \colon \shVO \to \Omega_{X-D}^1
\tensor \shVO$ on the underlying holomorphic vector bundle $\shVO$ has regular
singularities; because $X$ is projective algebraic, it follows from \cite{Deligne}
that $\shVO$ and $\nabla$ are algebraic.  Admissibility implies that each Hodge
bundle $F^p \shVO$ is an algebraic subbundle of $\shVO$; note that they satisfy
$\nabla \bigl( F^p \shVO \bigr) \subseteq \Omega_{X-D}^1 \tensor F^{p-1} \shVO$
because of Griffiths transversality.

To prove the constructibility of $I(\nu)$, our starting point is the exact sequence
\begin{equation} \label{eq:classical}
	0 \to F^0 \shHO \to F^0 \shVO \to \shO \to 0
\end{equation}
of algebraic vector bundles on $X-D$. Let $U$ be any affine Zariski open subset of
$X-D$ with the following two properties: (1) both $F^0 \shHO$ and $F^{-1}
\shHO / F^0 \shHO$ restrict to trivial bundles on $U$; (2) there are coordinates
$x_1, \dotsc, x_n \in \Gamma(U, \OU)$, where $\Gamma(U, \argbl)$ always denotes the
space of all algebraic sections of an algebraic coherent sheaf. Since $X-D$
can be covered by finitely many such open subsets, it is clearly sufficient to show that
$I(\nu) \cap TU$ is a constructible subset of $TU$.

By our choice of $U$, the tangent bundle $TU$ is trivial; let $\xi_1,
\dotsc, \xi_n \in \Gamma(TU, \shO_{TU})$ be the coordinates in the fiber direction
corresponding to the algebraic vector fields $\partial/\partial x_1, \dotsc,
\partial/\partial x_n$. Let $q = \rk F^0 \shHO$ and $p = \rk F^{-1} \shHO \geq q$; we
can then choose algebraic sections $e_1, \dotsc, e_p \in \Gamma(U, F^{-1} \shHO)$
such that $e_1, \dotsc, e_q \in \Gamma(U, F^0 \shHO)$ are a frame for $F^0 \shHO$,
and $e_1, \dotsc, e_p$ are a frame for $F^{-1} \shHO$. For $i=1, \dotsc, q$, we get
\[
	\nabla e_i = \sum_{k=1}^n \sum_{j=1}^p dx_k \tensor a_{i,j}^k e_j
\]
with certain functions $a_{i,j}^k \in \Gamma(U, \OU)$. Let $\nut \in \Gamma(U, F^0
\shVO)$ be any lifting of the element $1 \in \Gamma(U, \OX)$; then $\nabla \nut \in
\Gamma(U, \Omega_U^1 \tensor F^{-1} \shHO)$ can be written in the form
\[
	\nabla \nut = \sum_{k=1}^n \sum_{j=1}^p dx_k \tensor f_j^k e_j
\]
for certain functions $f_j^k \in \Gamma(U, \OU)$. By definition, a point $(x, \xi)
\in TU$ lies in the zero locus $I(\nu)$ of the infinitesimal invariant iff there are
holomorphic functions $\varphi_1, \dotsc, \varphi_q$ defined in a small open ball
around $x \in U$, such that
\[
	\nabla \left( \nut - \sum_{i=1}^q \varphi_i e_i \right)
\]
vanishes at the point $(x, \xi)$. When expanded, this translates into the condition
that
\[
	\sum_{k=1}^n \sum_{j=1}^p \xi_k f_j^k(x) e_j = 
		\sum_{i=1}^q \sum_{k=1}^n \xi_k \frac{\partial \varphi_i}{\partial x_k}(x) e_i
		+ \sum_{i=1}^q \varphi_i(x) \sum_{k=1}^n \sum_{j=1}^p \xi_k a_{i,j}^k(x) e_j.
\]
This is a system of $p$ linear equations in the $q(n+1)$ complex numbers 
\[
	\varphi_i(x) \quad \text{and} \quad \frac{\partial \varphi_i}{\partial x_k}(x)
		\qquad \text{(for $1 \leq i \leq q$ and $1 \leq k \leq n$),}
\]
with coefficients in the ring $\Gamma(TU, \shO_{TU})$ of regular functions on $TU$.
The proof of Proposition~\ref{prop:zero-locus} shows that the set of points $(x, \xi)
\in TU$, where this system has a solution, is a constructible subset of $TU$. This
completes the proof of Theorem~\ref{thm:main}.

\subsection{A more sophisticated description}

For some purposes, it is better to have a natural extension of $I(\nu)$ to the entire
cotangent bundle $TX$, without modifying the ambient variety $X$. In this section, we
indicate how such an extension can be constructed using the theory of mixed Hodge
modules \cite{Saito-MHM}. 

We begin by recalling how one associates a short exact sequence of the form
\begin{equation} \label{eq:extension}
	0 \to F_0 \Mmod \to F_0 \Nmod \to \OX \to 0
\end{equation}
to the given admissible normal function; here $F_0 \Mmod$ and $F_0 \Nmod$ are
algebraic coherent sheaves on $X$, and all three morphisms are morphisms of 
algebraic coherent sheaves.

The polarizable variation of Hodge structure $\shH$ extends uniquely to a 
polarizable Hodge module with strict support equal to $X$. We denote by $\Mmod$
the underlying regular holonomic $\Dmod_X$-module; it is the minimal extension 
of the flat vector bundle $\shHO$. It has a good filtration $F_{\bullet} \Mmod$ 
by $\OX$-coherent subsheaves, and $F_k \Mmod$ is an extension of the Hodge 
bundle $F^{-k} \shHO$. Since $X$ is a complex projective variety, each 
$F_k \Mmod$ is an algebraic coherent sheaf, and $\Mmod$ is an algebraic 
$\Dmod_X$-module.

Because the normal function $\nu$ is admissible, the corresponding variation of
mixed Hodge structure extends uniquely to a mixed Hodge module on $X$; in fact,
this condition is equivalent to admissibility \cite{Saito-ANF}*{p.~243}. 
Let $\Nmod$ denote the underlying regular holonomic $\Dmod_X$-module, and 
$F_{\bullet} \Nmod$ its Hodge filtration; as before, $\Nmod$ is an algebraic 
$\Dmod_X$-module, and each $F_k \Nmod$ is an algebraic coherent sheaf.  We have
an exact sequence of regular holonomic $\Dmod_X$-modules
\[
	0 \to \Mmod \to \Nmod \to \OX \to 0,
\]
in which all three morphisms are strict with respect to the Hodge filtration;
in particular, \eqref{eq:extension} is an exact sequence of algebraic coherent 
sheaves on $X$. Because $\Nmod$ is a filtered $\Dmod_X$-module, we have 
$\CC$-linear morphisms $\shT_X \tensor F_k \Nmod \to F_{k+1} \Nmod$; note that 
they are not $\OX$-linear.

We can use the exact sequence in \eqref{eq:extension} to construct an extension
of the zero locus $I(\nu)$ to all of $X$. Inside the tangent bundle $TX$, we
define a subset 
\[
	\It(\nu) = \menge{(x, \xi) \in TX}%
		{\text{$(\xi \cdot \sigma)(x) = 0$ for some $\sigma \in F_0 
                         \Nmod$ with $\sigma \mapsto 1$}},
\]
where the notation ``$\sigma \in F_0 \Nmod$'' means that $\sigma$ is a 
holomorphic section of the sheaf $F_0 \Nmod$, defined in some open neighborhood
of the point $x \in X$. 

\begin{lemma}
We have $\It(\nu) = I(\nu)$ over the Zariski open subset of $X$ where the 
variation of Hodge structure $\shH$ is defined.
\end{lemma}

\begin{proof}
This is obvious from the definitions.
\end{proof}

Denote by $p \colon TX \to X$ the projection. The pullback $\pu T_X$ of the 
tangent sheaf has a tautological global section $\theta$; in local holomorphic 
coordinates $x_1, \dotsc, x_n$ on $X$, and corresponding coordinates 
$(x_1, \dotsc, x_n, \xi_1, \dotsc, \xi_n)$ on $TX$, it is given by the formula 
\[
	\theta(x_1, \dotsc, x_n, \xi_1, \dotsc, \xi_n) =
		\xi_1 \frac{\partial}{\partial x_1} + \dotsb 
                      + \xi_n \frac{\partial}{\partial x_n}.
\]
Let $\Mmodt$ denote the pullback of $\Mmod$ to a filtered $\Dmod$-module on the
tangent bundle; because $p$ is smooth, we have $\Mmodt = \pu \Mmod$ and 
$F_k \Mmodt = \pu F_k \Mmod$. Similarly define $\Nmodt$. 

\begin{lemma}
In the notation introduced above, we have
\[
	\It(\nu) = \menge{(x, \xi) \in TX}%
		{\text{$(\theta \cdot \sigmat)(x, \xi) = 0$ for some 
                       $\sigmat \in F_0 \Nmodt$ with $\sigmat \mapsto 1$}}.
\]
\end{lemma}

\begin{proof}
The set on the right-hand side clearly contains $\It(\nu)$. To prove that the 
two sets are equal, suppose that we have $(\theta \cdot \sigmat)(x, \xi) = 0$ 
for some holomorphic section $\sigmat$ of $F_0 \Nmodt$, defined in a 
neighborhood of the point $(x, \xi) \in TX$. Since 
$F_0 \Nmodt = \pu F_0 \Nmod$, we can write $\sigmat = \sum_k
f_k \cdot \pu \sigma_k$ for suitably chosen $f_k \in \shO_{TX}$ and 
$\sigma_k \in F_0 \Nmod$. Define $\sigma = \sum_k f_k(\argbl, \xi) \sigma_k$; 
then $\sigma \in F_0 \Nmod$ and $\sigma \mapsto 1$. A brief calculation in 
local coordinates shows that
\[
	(\xi \cdot \sigma)(x) = (\theta \cdot \sigmat)(x, \xi) = 0,
\]
and so we get $(x, \xi) \in \It(\nu)$ as desired.
\end{proof}

The next step is to show that $\It(\nu)$ is the zero locus of a holomorphic 
section of an analytic coherent sheaf on $TX$. Let $\shF$ denote the analytic 
coherent sheaf on $TX$ obtained by taking the quotient of $F_1 \Mmodt$ by the 
analytic coherent subsheaf generated by $\theta \cdot F_0 \Mmodt$. For any 
local holomorphic section
$\sigma \in F_0 \Nmodt$ with $\sigma \mapsto 1$, we have 
$\theta \cdot \sigma \in F_1 \Mmodt$, and the image of $\theta \cdot \sigma$ in
the quotient sheaf $\shF$ is independent of the choice of $\sigma$, due to the 
exactness of \eqref{eq:extension}. In this manner, we obtain a global 
holomorphic section $s$ of the sheaf $\shF$.

\begin{lemma} \label{lem:zero-locus}
$\It(\nu)$ is the zero locus of the section $s$ of the coherent sheaf $\shF$.
\end{lemma}

\begin{proof}
If $(x, \xi) \in \It(\nu)$, then we have $(\theta \cdot \sigmat)(x, \xi) = 0$ 
for some choice of $\sigmat \in F_0 \Nmodt$ with $\sigma \mapsto 1$; in 
particular, $s(x, \xi) = 0$. Conversely, suppose that we have $s(x, \xi) = 0$ 
for some point $(x, \xi) \in TX$. By definition of $\shF$, we can then find 
local sections $\sigmat_k \in F_0 \Mmodt$ and local holomorphic functions 
$f_k \in \shO_{TX}$, such that
\[
	\theta \cdot \sigmat - \sum_k f_k \theta \cdot \sigmat_k
\]
vanishes at the point $(x, \xi)$. Set $a_k = f_k(x, \xi) \in \CC$; then
\[
	\theta \cdot \left( \sigmat - \sum_k a_i \sigmat_k \right) =
	\theta \cdot \sigmat - \sum_k a_k \theta \cdot \sigmat_k
\]
also vanishes at $(x, \xi)$, and this shows that $(x, \xi) \in \It(\nu)$.
\end{proof}

Despite the analytic definition, both $\shF$ and $s$ are actually algebraic 
objects.

\begin{lemma} \label{lem:algebraic}
$\shF$ is an algebraic coherent sheaf on $TX$, and $s \in \Gamma(TX, \shF)$ is an
algebraic global section.
\end{lemma}

\begin{proof}
Each $F_k \Mmodt = \pu F_k \Mmod$ is an algebraic coherent sheaf on $TX$, and 
since the tautological section $\theta \in \Gamma(TX, \pu \shT_X)$ is clearly 
algebraic, it follows that $\shF$ is an algebraic coherent sheaf. To show that 
the global section $s \in \Gamma(TX, \shF)$ is algebraic, observe that we have an 
exact sequence of algebraic coherent sheaves
\[
	0 \to F_0 \Mmodt \to F_0 \Nmodt \to \shO_{TX} \to 0;
\]
indeed, \eqref{eq:extension} is exact, and $p \colon TX \to X$ is a smooth 
affine morphism. At every point $(x, \xi) \in TX$, we can therefore find an 
algebraic section $\sigma \in F_0 \Nmodt$, defined in a Zariski open 
neighborhood of $(x, \xi)$, such that $\sigma \mapsto 1$. This clearly implies 
that $s$, which is locally given by the image of $\theta \cdot \sigma$ in 
$\shF$, is itself algebraic.
\end{proof}
	
To prove Theorem~\ref{thm:main}, it is clearly sufficient to show that the set
$\It(\nu)$ is constructible in the Zariski topology on $TX$.
Lemma~\ref{lem:zero-locus} and Lemma~\ref{lem:algebraic} reduce the problem to 
the following general result in abstract algebraic geometry: On any algebraic 
variety, the zero locus of a section of a coherent sheaf is constructible (but 
not, in general, Zariski closed). This fact is certainly well-known, but since 
it was surprising to us at first, we have decided to include a simple proof in 
the following section.

\subsection{Zero loci of sections of coherent sheaves}

In this section, we carefully define the ``zero locus'' for sections of 
coherent sheaves, and show that it is always constructible in the Zariski
topology. This is obviously a local problem, and so it suffices to consider 
the case of affine varieties. Let $R$ be a commutative ring with unit; to avoid
technical complications, we shall also assume that $R$ is Noetherian. For any 
prime ideal $\pp \subseteq R$, we denote by the symbol
\[
	\kappa(\pp) = R_{\pp} / \pp R_{\pp}
\]
the residue field at $\pp$; it is isomorphic to the field of fractions of the 
local ring $R_{\pp}$. Let $X = \Spec R$ be the set of prime ideals of the ring 
$R$, endowed with the Zariski topology. For any ideal $I \subseteq R$, the set
\[
	V(I) = \menge{\pp \in X}{\pp \supseteq I}
\]
is closed in the Zariski topology on $X$, and any closed subset is of this
form; likewise, for any element $f \in R$, the set
\[
	D(f) = \menge{\pp \in X}{\pp \not\ni f}
\]
is an open subset, and these open sets form a basis for the Zariski topology. 

\begin{definition}
A subset of $X$ is called \define{constructible} if it is a finite union of
subsets of the form $D(f) \cap V(I)$. 
\end{definition}

Here is how this algebraic definition is related to constructibility on
complex algebraic varieties. Suppose that $R$ is a $\CC$-algebra of finite
type. Let $X(\CC)$ be the set of all maximal ideals of $R$, endowed with the
classical topology; it is an affine complex algebraic variety, and the inclusion
mapping $X(\CC) \into X$ is continuous.

\begin{definition}
A subset of $X(\CC)$ is called \define{constructible} (in the Zariski topology)
if it is the set of maximal ideals in a constructible subset of $X$.
\end{definition}

Any coherent sheaf on $X = \Spec R$ is uniquely determined by the finitely 
generated $R$-module of its global sections; conversely, any finitely generated
$R$-module $M$ defines a coherent sheaf on $X$, and hence by restriction to the
subset $X(\CC)$ an algebraic coherent sheaf $\shF_M$ on $X(\CC)$. Its fiber at 
the point corresponding to a maximal ideal $\mm \subseteq R$ is the 
finite-dimensional $\CC$-vector space
\[
	M \tensor_R \kappa(\mm) = M_{\mm} / \mm M_{\mm}.
\]
Similarly, any element $m \in M$ defines an algebraic global section $s_m$ of 
the sheaf $\shF_M$. Obviously, $s_m$ vanishes at the point corresponding to a
maximal ideal $\mm \subseteq R$ if and only if $m$ goes to zero in 
$M \tensor_R \kappa(\mm)$.  Thus if we define
\[ 
	Z(M, m) = \menge{\pp \in X}%
		{\text{$m$ goes to zero in $M \tensor_R \kappa(\pp)$}},
\]
then the zero locus of $s_m$ on $X(\CC)$ is precisely the set of maximal ideals
in $Z(M, m)$. Thus the desired result about zero loci of sections of coherent 
sheaves is a consequence of the following general theorem in commutative 
algebra.

\begin{proposition} \label{prop:zero-locus}
Let $R$ be a commutative Noetherian ring with unit. Then for any finitely generated
$R$-module $M$, and any $m \in M$, the set $Z(M,m)$ is constructible.
\end{proposition}

\begin{proof}
We are going to construct a finite covering
\[
	\Spec R = \bigcup_{k=1}^n D(f_k) \cap V(I_k)
\]
with $f_1, \dotsc, f_n \in R$ and $I_1, \dotsc, I_n \subseteq R$, such that
for every $k=1, \dotsc, n$, one has
\[
	Z(M, m) \cap D(f_k) \cap V(I_k) = D(f_k) \cap V(I_k + J_k),
\]
for a certain ideal $J_k \subseteq R$. This is sufficient, because it implies that
\[
	Z(M, m) = \bigcup_{k=1}^n D(f_k) \cap V(I_k + J_k)
\]
is a constructible subset of $\Spec R$.

Since $M$ is finitely generated and $R$ is Noetherian, we may find a 
presentation
\begin{equation} \label{eq:presentation}
	R^{\oplus q} \xrightarrow{\,A\,} R^{\oplus p} \onto M,
\end{equation}
in which $A$ is a $p \times q$-matrix with entries in $R$. Let 
$y \in R^{\oplus p}$ be any vector mapping to $m \in M$. Then $Z(M, m)$ is the \
set of $\pp \in \Spec R$ such that the equation $y = A x$ has a solution over 
the field $\kappa(\pp)$. 

We construct the desired covering of $\Spec R$ by looking at all possible minors of
the matrix $A$. Fix an integer $0 \leq \ell \leq \min(p,q)$ and an $\ell \times
\ell$-submatrix of $A$; to simplify the notation, let us assume that it is the $\ell
\times \ell$-submatrix in the upper left corner of $A$. Let $f$ be the determinant of
the submatrix, and let $I$ be the ideal generated by all minors of $A$ of size
$(\ell+1) \times (\ell+1)$; if $\ell = 0$, we set $f = 1$, and if $\ell = \min(p,q)$,
we set $I = 0$. We can then make a coordinate change in $R^{\oplus q}$, invertible
over the localization $R_f = R \lbrack f^{-1} \rbrack$, and arrange that
\[
	A = \begin{pmatrix}
		f & 0 & \cdots & 0 & 0 & \cdots & 0 \\
		0 & f & \cdots & 0 & 0 & \cdots & 0 \\
		\vdots & \vdots & & \vdots & \vdots & & \vdots \\
		0 & 0 & \cdots & f & 0 & \cdots & 0 \\
		a_{\ell+1,1} & a_{\ell+1,2} & \cdots & a_{\ell+1, \ell} 
			& a_{\ell+1,\ell+1} & \cdots & a_{\ell+1,q} \\
		\vdots & \vdots & & \vdots & \vdots & & \vdots \\
		a_{p,1} & a_{p,2} & \cdots & a_{p, \ell} 
			& a_{p,\ell+1} & \cdots & a_{p,q} 
	\end{pmatrix} \in \Mat_{p \times q}(R).
\]
Now let $J \subseteq R$ be the ideal generated by the elements
\[
	f y_i - \sum_{j=1}^{\ell} a_{i,j} y_j
\]
for $i = \ell+1, \dotsc, p$. Then we have 
\[
	Z(M, m) \cap D(f) \cap V(I) = D(f) \cap V(I + J).
\]
Indeed, suppose that $\pp$ is any prime ideal with $f \not\in \pp$ and $I \subseteq
\pp$. Since $a_{i,j} \in \pp$ for every $\ell+1 \leq i \leq p$ and $\ell+1 \leq j \leq q$,
the equation $y = Ax$ reduces over the field $\kappa(\pp)$ to the equations $y_i = f
x_i$ for $i=1, \dotsc, \ell$, and 
\[
	y_i = \sum_{j=1}^{\ell} a_{i,j} x_j
\]
for $i = \ell+1, \dotsc, p$; they are obviously satisfied if and only if $J
\subseteq \pp$.

We now obtain the assertion by applying the above construction of $f$, $I$, and
$J$ to all possible $\ell \times \ell$-submatrices of $A$.
\end{proof}

Here is a simple example to show that, when the coherent sheaf is not locally 
free, the zero locus of a section need not be Zariski closed.

\begin{example}
Let $R = \CC \lbrack x, y \rbrack$, let $M$ be the ideal of $R$ generated by 
$x,y$, and let $m = x$. Then $M$ has a free resolution of the form 
$R \to R^{\oplus 2}$, and $\pp \in Z(M,m)$ if and only if the equations 
$1 + yf = 0$ and $xf = 0$ have a common solution $f \in \kappa(\pp)$. A simple 
computation now shows that
\[
    Z(M,m) = \menge{\pp \in \Spec R}{\text{$x \in \pp$ and $y \not\in \pp$}}.
\]
As a subset of $\CC^2$, the zero locus consists of the $y$-axis minus the 
origin; it is constructible, but not Zariski closed.
\end{example}

\section{Relation to algebraic cycles}
\label{sec:GG}

\def\C{\mathbb C}
\def\a{\alpha}
\def\b{\beta}
\def\g{\gamma}

\subsection{Green-Griffiths Program} Our interest in the algebraicity of 
$I(\nu)$ is motivated in part by the program~\cite{GG} of Green and Griffiths 
to study the Hodge conjecture via singularities of normal functions.  More 
precisely, given a smooth complex projective variety $X$, a very ample line 
bundle $L\to X$ and a non-torsion, primitive Hodge class $\zeta$ of type 
$(n,n)$ on $X$, Griffiths and Green construct an admissible normal function
$$
        \nu_{\zeta}:P-\hat X\to J(\mathcal H)
$$
on the complement of the dual variety $\hat X$ in 
$P=\mathbb P H^0(X,\mathcal O(L))$.  At each point $\hat x\in\hat X$, the
cohomology class of $\nu_{\zeta}$ localizes to an invariant
$$
      \sing_{\hat x}(\nu_{\zeta}) \in \IH_{\hat{x}}^1(\mathcal H)
$$
called the \define{singularity} of $\nu_{\zeta}$ at $\hat x$.  A normal function
$\nu_{\zeta}$ is said to be \define{singular} if there is a point $\hat x\in\hat X$
at which $\sing_{\hat x}(\nu_{\zeta})$ is non-torsion.

\begin{conjecture}\label{conj:main} Let $(X,L,\zeta)$ be as above.  Then, 
there exists an integer $k>0$ such that after replacing $L$ by $L^k$, the 
associated normal function $\nu_{\zeta}$ is singular.
\end{conjecture}

\begin{theorem}\label{thm:GG}\cites{GG,BFNP,dCM} Conjecture \eqref{conj:main}
holds (for every even dimensional $X$ and every non-torsion, primitive
middle dimensional Hodge class $\zeta$) if and only if the Hodge conjecture
holds (for all smooth projective varieties).
\end{theorem}

\par Now, as explained in part III of~\cite{GG} one can also define a
notion $\sing_{\hat x}(\delta\nu_{\zeta})$ of the singularities of infinitesimal 
invariant $\delta\nu_{\zeta}$ of $\nu_{\zeta}$.  Moreover,
$$
      \sing_{\hat x}(\delta\nu_{\zeta}) = \sing_{\hat x}(\nu_{\zeta})
$$
for $L\gg 0$.  As a first attempt at constructing points at which
$\nu_{\zeta}$ is singular, observe that 
$$
      Z(\nu_{\zeta}) = \{\, p\in P-\hat X \mid \nu_{\zeta}(p) = 0\,\}
$$
is an analytic subset of $P-\hat X$, and hence it is natural to ask
if its closure is an algebraic subvariety of $P$ which intersects
$\hat X$ at some point where $\nu_{\zeta}$ is singular.  An affirmative answer 
is provided by the following two results:

\begin{theorem}\label{thm:alg-zero}\cites{BP,KNU,Sch2} If $S$ is a smooth 
complex algebraic variety and $\nu:S\to J(\mathcal H)$ is an admissible normal
function then $Z(\nu)$ is an algebraic subvariety of $S$.
\end{theorem}

\begin{proposition}\label{prop:schnell}\cite{Sch1} Let $\nu_{\zeta}$ be the 
normal function on $P\setminus\hat X$, associated to a non-torsion primitive 
Hodge class $\zeta\in H^{2n}(X,\mathbb Z)\cap H^{n,n}(X)$. Assume that 
$Z(\nu_{\zeta})$ contains an algebraic curve $C$, and that $P = |L^d|$ for $L$ 
very ample and $d\geq 3$.  Then $\nu_{\zeta}$ is singular at one of the points 
where the closure of $C$ meets $\hat X$.
\end{proposition}

\par The caveat here, which is illustrated in the example~\eqref{ex:quadric} 
below, is that there is no reason for $Z(\nu_{\zeta})$ to contain a curve.
The advantage of working with the infinitesimal invariant is that it is
often easier to compute \cite{Griffiths}, and will vanish along directions tangent to 
$Z(\nu)$.  Of course, $I(\nu)$ will also contain the directions tangent
to any m-torsion locus of $\nu$, as well as potentially other components.

\begin{question} Is there an analog of Proposition~\eqref{prop:schnell} for 
$I(\nu_{\zeta})$?
\end{question}

\begin{remark} The study of zero loci of normal function also arises in
connection with the construction of the Bloch--Beilinson filtration on
Chow groups.  For a survey of results of this type, see~\cite{KP}.
\end{remark}

\par The determination of a good notion of the expected dimension of the
zero locus of a normal function is an important open problem in the study
of algebraic cycles.  In particular, in the Green--Griffiths setting, if a 
smooth projective variety has moduli, any reasonable expected dimension count 
is probably only valid at the generic point of the locus where the class
$\zeta$ remains a Hodge class.

\par In the case of a smooth projective surface $X$, if $L=\mathcal O(D)$ is a 
very ample line bundle, then a Riemann--Roch calculation shows the expected
dimension of the zero locus of the associated  normal functions arising from 
the Green--Griffiths program (i.e., comparing the dimensions of the fiber and 
the base) is
$$
     -(D\cdot K_X) + \chi(\mathcal O_X) - 2
$$
where $K_X$ is the canonical bundle of $X$.  For $X$ of general type, on the
basis of this calculation one would expect the zero locus to be empty for
all sufficiently ample $L$.  We close with a careful study of a simple
example of normal function of Green--Griffiths type for which the naive
expected dimension count is positive.

\begin{example}\label{ex:quadric} Let $X=\mathbb P^1\times\mathbb P^1$ viewed 
as the smooth quadric $Q=V(q)\subseteq\mathbb P^3$ defined by the vanishing of 
$
      q = x_0^2 + x_1^2 + x_2^2 + x_3^2
$.
Let $L_{\a}$ and $L_{\b}$ be the lines on $Q$ defined by the equations
$$
     L_{\a} : t\mapsto [1,t,it,i],\qquad
     L_{\b} : t\mapsto [1,t,-it,i]
$$
Then, the difference $\zeta = [L_{\a}]-[L_{\b}]$ is a primitive Hodge class on 
$X$.  For future use, we also introduce the line
$$
     L_{\gamma} : t\mapsto [1,t,-it,-i]
$$
which is parallel to $L_{\alpha}$ and intersects $L_{\beta}$ at $t=\infty$.

\par Let $P = \mathbb P H^0(X,\mathcal O(2))$.  Then, the associated normal
function $\nu_{\zeta}$ assigns to each smooth section 
$$
         X_{\sigma} = V(q)\cap V(\sigma)
$$
the class of 
$
      (L_{\a}-L_{\b})\cap V(\sigma)
$ 
in the Jacobian of $X_{\sigma}$.  A naive expected dimension count for the
zero locus of $\nu_{\zeta}$ can be obtained as follows:  The dimension
of $P$ is $8=10-1-1$ since the space of quadratic forms on $\C^4$ has
dimension $10$, and we need mod out by $Q$ and then projectivize.  The
adjunction formula shows the fibers $X_{\sigma}$ to have genus $1$.  
Accordingly, the graph of $\nu_{\zeta}$ in the associated bundle of
Jacobians $J\to P$ has codimension $1$.  Likewise, the zero section
of $J$ is also codimension $1$, and so to first approximation the
zero locus of $\nu_{\zeta}$ in this case should have codimension $2$
in $J$, which corresponds to a $7$-dimensional subvariety of $P$.

\par To see that the zero locus of $\nu_{\zeta}$ is in fact empty, let 
$Y\subset\mathbb P^3$ be a smooth quadric which intersects $X$ 
in a smooth curve $E$.  Let $\Lambda\subset X$ be a line of the form
$\{z\}\times\mathbb P^1$ which intersects $E$ in a pair of distinct 
points
$$
       e = (z,w),\qquad f = (z,w')
$$
Let the line $\Upsilon = \mathbb P^1\times\{w\}$ intersect $E$ in the divisor
$e+g$.  Then, since every line on $X$ is parallel to either 
$\Lambda$ or $\Upsilon$, it follows that $L_{\a}-L_{\b}$ intersects $E$
in a divisor which is linearly equivalent to 
$$
      (e+f) - (e+g) \sim f-g
$$
Accordingly, if $\nu_{\zeta}$ vanishes at $Y$ then $f\sim g$ and hence
$\Lambda = \Upsilon$.

\par As a consequence of symmetries however, the $2$-torsion locus of 
$\nu_{\zeta}$ is non-zero.  To be explicit, let 
$S = \mathbb C - \{-1,-i,0,i,1\}$ and $\mu:S\to P$ be the map which associates 
to a point $s\in S$ the quadric
$$
      Q_s = V(s^2 x_0^2 + x_1^2 - x_2^2 - s^2 x_3^2)
$$ 
Then, for each $s\in S$, the associated curve $X_{\mu(s)}$ is smooth.

\par Let $\theta$ be the involution of $\mathbb P^3$ induced by 
the linear map
$$
      (c_0,c_1,c_2,c_3) \mapsto (-c_3,-c_2,c_1,c_0)
$$
on $\mathbb C^4$.  Then, the lines $L_{\alpha}$ and $L_{\gamma}$ are the 
projectivizations of the $\pm i$-eigenspaces of this map, and hence are 
pointwise fixed under the action of $\theta$.  The involution $\theta$ also 
fixes the quadrics $Q$ and $Q_s$, and hence the curve $X_{\mu(s)}$.  
Consequently, the fixed points of the action $\theta$ on $X_{\mu(s)}$ are 
exactly the four points
$$
\aligned
       \a_1 &= [1,is,-s,i],\qquad \a_2 = [1,-is,s,i]  \\
       \g_1 &= [1,is,s,-i],\qquad \g_2 = [1,-is,-s,-i]
\endaligned
$$
corresponding to the intersection of the lines $L_{\a}$ and $L_{\g}$ with $Q_s$. 
The line $L_{\beta}$ on the other hand intersects $Q_s$ at the points
$$
       \b_1 =  [1,is,s,i],\qquad \b_2 = [1,-is,-s,i] 
$$
which are interchanged under the action of $\theta$.

\par Let 
$$
\aligned
     F_1 &= s x_0 + i x_1 - x_2 + i s x_3 \\
     F_2 &= s x_0 -i x_1 + x_2 + i s x_3  \\
     F_3 &= ix_1 + x_2
\endaligned
$$
Then, direct calculation shows that $V(F_1)$ is a plane passing through
$\{\a_1,\a_2,\g_1\}$ which is also tangent to $E_s$ at $\g_1$.  Similarly,
$V(F_2)$ is a plane passing through $\{\a_1,\a_2,\g_2\}$ which is tangent
to $E_s$ at $\g_2$.  Finally, $V(F_3)$ is a plane passing through 
$\{\b_1,\b_2,\g_1,\g_2\}$.  Moreover, one can easily check that these
planes have no additional points of intersection or tangency other than
the ones listed above.  Therefore, the rational function
$$
        F = (F_1 F_2)/F_3^2
$$
on $\mathbb P^3$ restricts to a meromorphic function on $E_s$ with divisor
$$
     (\a_1 + \a_2 + 2 \g_1) + (\a_1 + \a_2 + 2 \g_2) - 
     2(\b_1+\b_2+\g_1+\g_2) = 2(\a_1 + \a_2) - 2(\b_1+\b_2)
$$
and hence $2\nu_{\zeta}$ vanishes along the image of $\mu$.
\par Finally, to get a $7$-dimensional subvariety of $P$ as predicted above,
observe that the group $SO(4)$ has dimension $6$ and acts on $\mathbb P^3$
preserving the quadric $Q$.  This action also fixes the integral Hodge
class $\zeta$, and hence acts on the $2$-torsion locus. The orbit of
$S$ under the action of $SO(4)$ therefore provides a $7$-dimensional
complex analytic subvariety of $P$ on which $2\nu_{\zeta}$ vanishes.
\end{example}

\begin{remark}
The infinitesimal invariant for the intersection of a generic quadric $X$ and cubic
$Y$ in $\PPn{3}$ is considered in \cite{Griffiths}*{Section~6d}, where it is shown
that in this case, the invariant determines the curve $C = X \cap Y$.
\end{remark}

\end{document}